\documentclass{amsart}

\usepackage[english]{babel}

\usepackage{mathrsfs, mathtools, amssymb}

\usepackage{dsfont}

\usepackage[paper=a4paper, margin=3cm]{geometry}

\usepackage[breaklinks]{hyperref}       
\usepackage{url}            
\usepackage{booktabs}       
\usepackage{amsfonts}       
\usepackage{nicefrac}       
\usepackage{microtype}      
\usepackage{lipsum}

\usepackage{amsmath,amssymb,amsthm,amscd}
\usepackage{amsrefs}
\usepackage{tikz,tikz-cd}
\usetikzlibrary{arrows,arrows.meta}

\usepackage{graphicx}
\graphicspath{}
\DeclareGraphicsExtensions{.pdf,.png,.jpg,.jpeg}

\usepackage{yhmath}
\usepackage{enumitem}
\usepackage{breakurl}

\newcounter{stmcounter}[section]

\newcounter{thmMaincounter}

\newtheorem{formula}{}[section]
\newtheorem{proposition}[formula]{Proposition}
\newtheorem{corollary}[formula]{Corollary}
\newtheorem{lemma}[formula]{Lemma}
\newtheorem{theorem}[formula]{Theorem}
\newtheorem{theoremM}[thmMaincounter]{Theorem}

\theoremstyle{definition}

\newtheorem{definition}[formula]{Definition}
\newtheorem{example}[formula]{Example}

\theoremstyle{remark}

\newtheorem{remark}[formula]{Remark}

\newcommand{\Z}{\mathbb Z}
\newcommand{\R}{\mathbb R}

\newcommand{\Q}{\mathbb Q}
\newcommand{\N}{\mathbb N}
\newcommand{\F}{\mathbb F}

\newcommand{\lb}{\lbrace}
\newcommand{\rb}{\rbrace}
\newcommand{\la}{\langle}
\newcommand{\ra}{\rangle}

\DeclareMathOperator{\Tor}{Tor}

\DeclareMathOperator{\st}{star}

\DeclareMathOperator{\hoc}{hocolim}

\DeclareMathOperator{\Hm}{Hom}

\DeclareMathOperator{\Top}{Top}

\DeclareMathOperator{\Ext}{Ext}

\DeclareMathOperator{\srep}{srep}

\DeclareMathOperator{\chr}{char}
\DeclareMathOperator{\Tot}{Tot}

\renewcommand{\leq}{\leqslant}
\renewcommand{\geq}{\geqslant}
\renewcommand{\ge}{\geqslant}
\renewcommand{\le}{\leqslant}
\renewcommand{\phi}{\varphi}
\newcommand{\ot}{\leftarrow}

\newcommand{\Cc}{\mathbf{C}}
\newcommand{\Dd}{\mathbf{D}}
\newcommand{\Mm}{\mathbf{M}}

\newcommand{\Pp}{\mathbf{P}}
\newcommand{\Ab}{\mathbf{Ab}}
\newcommand{\sset}{\mathbf{sSet}}

\newcommand{\sab}{\mathbf{sAb}}
\renewcommand{\top}{\mathbf{Top}}

\def\k{\mathbf{k}}

\tikzcdset{arrow style=tikz, diagrams={>=stealth}}

\catcode`,\active

\catcode`\,12

\begin{document}
	
	\title[On formality of diagrams of Eilenberg-MacLane spaces]{On formality of diagrams of Eilenberg-MacLane spaces}
	
	\author{Grigory Solomadin}
	\address[G.\,Solomadin]{PU Marburg, Marburg, Germany}
	\email{grigory.solomadin@gmail.com}
	
	\author{Antoine Touz\'e}
	\address[A.\,Touz\'e]{Univ. Lille, CNRS, UMR 8524 - Laboratoire Paul Painlevé, F-59000 Lille, France}
	\email{antoine.touze@univ-lille.fr}
	
	
	\begin{abstract}
		In this paper, we establish formality (over $\Q$) for diagrams of Eilenberg-MacLane spaces of any height $n\geq 1$.
		This implies spectral sequence (over $\Q$) collapse at page $2$ for any diagram of EML spaces over any small category.
		We prove by functor calculus argument that formality does not hold over any fixed commutative ring $\k$ not containing $\Q$, where the category of diagrams is over the category generated by finite direct sums of a cyclic group.
	\end{abstract}
		
	\maketitle

\section{Introduction}

A space $X$ is called \emph{formal} over $\k$ if its singular chain complex $C_{*}(X;\k)$ with coefficients in $\k$ is connected by a zig-zag of quasi-isomorphisms to its singular homology $H_{*}(X;\k)$ (considered as a complex with zero differential). 
There are many variants of formality: one can consider singular chains or cochains, endow them with various additional structures (products, coproducts,\dots) and ask for a zig-zag of quasi-isomorphisms preserving the additional structures. 
More generally, one can consider \emph{formality of diagrams}. Namely, given a diagram of formal spaces, one can ask furthermore if the zig-zags of quasi-isomorphisms connecting the singular (co)chains of each object of the diagram to its (co)homology can be chosen in a way that is compatible with the maps of the diagram. If so, then one says that the diagram is formal (see definition~\ref{defi:formality} for a more precise formulation).

While formal spaces are not rare, formal diagrams of spaces are much more exceptional. For example, in rational homotopy theory, it is known that spheres, K\"ahler manifolds, Eilenberg-MacLane spaces, and many others are formal~\cite{FHT}. On the contrary, every map can be interpreted as a small diagram with one arrow and two vertices. By inspecting the Sullivan models, one can then see that the Hopf fibration $S^3\to S^2$ is not formal. More generally, it is known that a homologically trivial formal map must be nullhomotopic~\cite[Lemme~3.9]{le-si-81}.

In this article, we study the formality of diagrams of Eilenberg-MacLane spaces (EML spaces, for short). Every Eilenberg-MacLane space can be given the structure of a strict abelian group, and every diagram can be rigidified into a diagram in which the arrows preserve this abelian group structure (see proposition~\ref{pr:zigzagofdiagramsofEML}). This abelian group structure on spaces induces a commutative differential graded Hopf algebra structure on the level of singular chains, therefore we can ask about formality in this context. 
In the present note, we establish formality for any diagram of Eilenberg-MacLane spaces with fixed height $n\geq 1$ over any $\Q$-algebra $\k$ (see theorem~\ref{thm:formQ}).
\begin{theoremM}\label{thmm:formQ}
	For any $n\ge 1$ let $D\colon\Cc\to \sset_{*}$ be any diagram of EML spaces of height $n$.
	Let $\k$ be a $\Q$-algebra.
	Then the diagram $c\mapsto C_{*}(D(c);\k)$ is formal in the category of augmented commutative differential graded $\k$-algebras if $n=1$, and it is formal in the category of commutative differential graded Hopf $\k$-algebras if $n\ge 2$.     
\end{theoremM}
The proof of formality relies on simplicial models for EML spaces using iterated bar $\overline{B}$ and $\overline{W}$ constructions~\cite{gu-ma-74}. If $n=1$ or $2$, we actually extend theorem \ref{thmm:formQ} to more general rings $\k$ in theorem~\ref{thm:formsmalldim}, provided the homotopy groups of the EML-spaces of the diagram are ``not too large with respect to $\k$''. For example, formality holds with coefficients $\k=\Z[1/k!]$, if one restricts to diagrams of Eilenberg-MacLane spaces with $\pi_{n}$ of $\k$-dimension $\leq k$, see example~\ref{ex:kfact}.

The formality results given in theorems \ref{thm:formQ} and \ref{thm:formsmalldim} seem to be folklore albeit not published in this generality to the knowledge of the authors (the case of tori was addressed in~\cite{fr-21}, and the integral formality of certain diagrams of EML spaces was addressed in~\cite{no-ra-05}).

We review an application of formality to the study of homotopy colimits in section~\ref{sec:app}. Namely, for all diagrams of spaces $D$ with shape $\Cc$ there are two first quadrant spectral sequences with the second pages given by (co)homology of the category $\Cc$ with coefficients in the singular (co)homology of $D$, and converging to the singular (co)homology of the homotopy colimit of $D$:
\[
	E^{2}_{p,q}(D)=H_{p}(\Cc;H_{q}(D;\k))\Rightarrow H_{p+q}(\hoc D;\k),
\]
\[
	E_{2}^{p,q}(D)=H^{p}(\Cc;H^{q}(D;\k))\Rightarrow H^{p+q}(\hoc D;\k).
\]
As a consequence of theorem~\ref{thmm:formQ}, we prove the following result (and actually a bit more, by providing an explicit model for the chains of the homotopy colimit) in theorem~\ref{thm:rcollf}.	

\begin{theoremM}\label{thmm:rcollf}
    Assume that $\k=\Q$. Let $n$ be a positive integer, and let $D$ be a diagram of pointed spaces. If every space $D(c)$ of the diagram is a nilpotent space with an abelian fundamental group, such that $\pi_k(D(c))\otimes\Q=0$ for $k\ne n$, then both spectral sequences collapse at the second page.
\end{theoremM}

However, general  formality statements are quite limited outside the cases singled out in theorems~\ref{thm:formQ} and~\ref{thm:formsmalldim}. 
Indeed, compact tori are EML spaces of height $1$, and computations of toric varieties already show that diagrams of tori may not be formal over $\k=\Z$. Namely, assuming natural formality of diagrams of tori over $\k=\Z$ or $\Z/2\Z$, one deduces the algebra isomorphism~\cite{fr-06}
\[
	H^{*}(X_{\Sigma};\k)\cong \Tor_{H^*(BT;\k)}(\k[\Sigma],\k),
\]
for the toric variety $X_{\Sigma}$ with the $T$-action corresponding to any smooth rational polyhedral fan $\Sigma$ in $\R^{n}$.
On the right hand side, the product is defined canonically via Koszul resolution.
However, the products are in general different by~\cite[Example~1.1]{fr-21}. Thus the singular cochains of toric diagrams cannot be formal as commutative differential graded algebras over neither of $\Z,\Z/2\Z$.
We remark that this argument does not forbid formality over $\Z[1/2]$ (since the products then coincide by~\cite[Theorem~1.2]{fr-21}), for which one can use theorem~\ref{thmm:nonformality} below instead. 

In section~\ref{sec-nonformalityFp} we give another obstruction to formality relying on some $\Ext$-computations in functor categories. We show that formality of singular chains may fail for some diagrams of EML spaces as soon as $\k$ is not a $\Q$-algebra, whatever the height of EML spaces of the diagram is. To be more specific, if $C$ is a cyclic group, we denote by $\langle C\rangle$ the full subcategory of abelian groups containing all the finite direct sums of copies of $C$. We prove the following result in theorem~\ref{thm:nonformality}.
\begin{theoremM}\label{thmm:nonformality}
	Assume that $\k$ is not a $\Q$-algebra, and let $C$ be a cyclic group such that $C\otimes_{\Z}\k\ne 0$. Assume that $D$ is a diagram of EML spaces of height $n\geq 1$ such that $\pi_n D(G)$ is naturally isomorphic to $G$. Then the diagram $c\mapsto C_*(D(c);\k)$ is not formal in the category of chain complexes of $\k$-modules.
\end{theoremM}

Using models for singular chains given by iterated bar constructions $\overline{B}$, we reduce the proof of theorem~\ref{thmm:nonformality} to non-formality of the complex $\overline{B}^n\k[G]$ as a functor of the abelian group $G$. In order to prove this, we show that the total ranks (with respect to the usual bigrading) of the groups 
\begin{equation}\label{eq:twoext}
	\Ext_{\mathcal{F}(\la C\ra,\k)}(\overline{B}^n\k[-],-\otimes \mathbb{F}),\ 
	\Ext_{\mathcal{F}(\la C\ra,\k)}(H_{*}(\overline{B}^n\k[-]),- \otimes\mathbb{F}),
\end{equation}
are different in the functor category $\mathcal{F}(\la C\ra,\k)$ of functors from $\langle C\rangle$ to $\k$-modules.
Here, $\mathbb{F}$ denotes any quotient field, $\chr\mathbb{F}=p>0$, of $\k$ such that $C\otimes \mathbb{F}\simeq \mathbb{F}$ holds.
Namely, we prove that the first group in~\eqref{eq:twoext} is isomorphic to $\mathbb{F}$, whereas the second group in~\eqref{eq:twoext} has rank $>1$. (Formality property would imply that both total ext groups are isomorphic, leading to a contradiction.)

In conclusion, let us compare our work with some recent results obtained by Zakharov. In particular,~\cite[Theorem~1.1.2]{za-25} implies that the singular chain functors $A\mapsto C_*(K(A,n);\k)$ are isomorphic to the Dold-Puppe derived functors $A\mapsto L_*S(A\otimes_\Z\k,n)$, and the latter are known to be formal under the hypotheses of our theorems~\ref{thm:formQ} and theorem~\ref{thm:formsmalldim}. However all these isomorphisms are isomorphisms in the category $\mathcal{F}(\Ab,D(\k))$ of functors from abelian groups to the derived category of $\k$-modules. On the contrary, all our results can be reformulated as statements in the derived category $D(\mathcal{F}(\Ab,\k))$ of functors from abelian groups to $\k$-modules, which is more rigid, and which is the correct category to consider if one wants to obtain collapsing of spectral sequences as in theorem~\ref{thmm:rcollf}. 
    	
\section{Definitions and basic facts on formality}\label{sec:defform}
	
\subsection{Eilenberg-MacLane spaces}
Let $n$ be a positive integer.
An \emph{Eilenberg-MacLane space  of height $n$} (or \emph{EML space of height $n$}, for short) is a pointed simplicial set $E$ such that $\pi_kE=0$ for $k\ne n$. An \emph{abelian Eilenberg-MacLane space of height $n$} (or abelian EML space of height $n$ for short) is a simplicial abelian group $E$ such that $\pi_kE=0$ for $k\ne n$. 

For all pointed simplicial sets $X$ we denote by $\mathbb{Z}[X]$ the free simplicial abelian group on $X$ and we let $X^{(n)}=\mathrm{cosk}_{n+1}\mathbb{Z}[X]/\mathbb{Z}[*]$, where $\mathrm{cosk}_{n+1}$ refers to the $(n+1)$-st coskeleton of a simplicial set~\cite{dk-84}. We observe that coskeletons of simplicial abelian groups are simplicial abelian groups, hence this construction actually defines a functor
$-^{(n)}:\sset_*\to \sab$.
Let $\alpha_{X}^{n}\colon X\to X^{(n)}$ denote the following composition of natural morphisms of simplicial sets (we abuse notations and write $X^{(n)}$ both for the simplicial abelian group and its underlying simplicial set)
\[
X\xrightarrow[]{h}\mathbb{Z}[X]\twoheadrightarrow \mathbb{Z}[X]/\mathbb{Z}[*] \twoheadrightarrow \mathrm{cosk}_{n+1}\mathbb{Z}[X]/\mathbb{Z}[*] =X^{(n)},
\]
where $h$ is the Hurewicz morphism~\cite[Chapter~8, Application~8.3.9]{we-97}. The next lemma allows to translate questions about EML spaces into questions about abelian EML spaces.

\begin{lemma}\label{lm:1}
	If $E$ is an EML space of height $n$, then $\alpha_{E}^{n}$ is a weak homotopy equivalence.     
\end{lemma}
\begin{proof}
	By construction, $\pi_k E^{(n)}=0$ if $k\ne n$. Therefore it suffices to check that $\pi_n(\alpha_{E}^{n})$ is an isomorphism, which follows from the Hurewicz theorem.
\end{proof}

Let $K:\mathbf{Ch}(\mathbb{Z})\to \sab$ denote the Dold-Kan functor~\cite[Chapter~8, \S8.4]{we-97}. For all abelian groups $A$, we let $A[n]$ be the chain complex equal to $A$ in degree $n$ and to zero in degrees different from $n$. Then the simplicial abelian group $K(A[n])$ is an abelian EML space with $\pi_n(A)\simeq A$. 
\begin{lemma}\label{lm:2}
	Every abelian EML space $E$ of height $n$ is connected to $K(\pi_nE[n])$ by a zig-zag of weak equivalences of simplicial abelian groups, which are natural with respect to all morphisms of the simplicial abelian group $E$.
\end{lemma}
\begin{proof}
	Using the Dold-Kan equivalence the lemma is equivalent to the well-known fact that every chain complex of abelian groups $C$ such that $H_k(C)=0$ for $k\ne n$ is connected to the chain complex $H_n(C)[n]$ by a zig-zag of quasi-isomorphisms natural with respect to $C$.
\end{proof}

We can apply the previous lemmas to prove a rigidification property for diagrams of EML spaces.
We first need to introduce a definition. We use the term diagram as a synonym of functor.

\begin{definition}
	A diagram $D:\Cc\to \sset_{*}$ is called a \emph{diagram of EML spaces (of height $n$)}
	if $D(c)$ is an EML space of height $n$ for all $c$.
	The diagram $D$ is called a \emph{diagram of abelian EML spaces (of height $n$)} if it factors as
	\[
	D\colon \Cc \xrightarrow[]{} \sab \xrightarrow[]{\mathrm{forget}}\sset_{*}. 
	\]  %
\end{definition}

Now applying lemmas~\ref{lm:1} and~\ref{lm:2}, we immediately obtain the following rigidification property. This property justifies that we only consider diagrams of abelian EML spaces in the remainder of section, as well as in section~\ref{sec-formality}.
\begin{proposition}\label{pr:zigzagofdiagramsofEML}
	Let $D:\Cc\to \sset_*$ be a diagram of EML spaces of height $n\ge 1$. Let $D':\Cc\to \sset_*$ denote the diagram of abelian EML spaces given by $D'(c)=K(\pi_nD(c)[n])$. Then $D$ is connected to $D'$ by a zigzag of natural transformations which are objectwise weak equivalences of simplicial sets.   
\end{proposition}

\subsection{Formality}
Let $\k$ be a commutative ring, and let $\Dd$ denote one of the following categories: 
\begin{enumerate}
	\item[(1)] the category
	$\mathrm{DGMod}_\k$ of complexes of $\k$-modules, 
	\item[(2)] the category $\mathrm{ACDGAlg}_\k$ of augmented graded commutative differential graded $\k$-algebras, 
	\item[(3)] the category $\mathrm{CDGHopf}_\k$ of graded commutative differential graded Hopf $\k$-algebras.
\end{enumerate}
A \emph{quasi-isomorphism} in $\Dd$ is a morphism which induces an isomorphism in homology. 
If $M$ is an object of $\Dd$, then we consider its homology $H_*(M)$ as an object of $\Dd$ with trivial differential. When doing this for Hopf algebras, we implicitly assume a flatness assumption on $H_*(M)$ to ensure that  $H_*(M^{\otimes 2})$ is isomorphic to $H_*(M)^{\otimes 2}$, hence to ensure that the coproduct on $M$ induces a coproduct on the homology.
\begin{definition}\label{defi:formality} 
	A diagram $D:\Cc\to \Dd$ is \emph{formal}
	if there is a zig-zag of quasi-isomorphisms in $\Dd$, natural with respect to $c$ in $\Cc$, between $D(c)$ and $H_*(D(c))$.
\end{definition}

The previous definition can be applied to diagrams of simplicial sets. To be more specific, for all commutative rings $\k$, we denote by $\k[X]$ the free simplicial $\k$-module on $X$ and by $C_*(X;\k)=C_{\mathrm{norm}}\k[X]$ the associated normalized complex of $\k$-modules, that is the complex of normalized singular chains of $X$. The homology of this complex is the singular homology of $X$, and we denote it by $H_*(X;\k)$.
When $X$ is a simplicial abelian group, the addition $X\times X\to X$ and the Eilenberg-Zilber shuffle map 
make $C_*(X;\k)$ into a graded commutative differential $\k$-algebra with unit induced by the inclusion of the constant simplicial abelian group $0$ into $X$. Moreover, the diagonal map $X\to X\times X$  and the Alexander-Whitney map 
induce a coproduct on $C_*(X,\k)$ with counit induced by the unique morphism of simplicial groups $X\to 0$. Together with the previous product, it endows $C_*(X;\k)$ with the structure of a graded commutative differential graded Hopf algebra over $\k$~\cite[Corollary~A.7]{gu-ma-74}. 
	
Thus, given a diagram $D:\Cc\to \sset_*$, we can interpret the assignment $c\mapsto C_*(D(c);\k)$ as a functor from $\Cc$ to $\mathrm{DGMod}_\k$, or to $\mathrm{ACDGAlg}_\k$ or to $\mathrm{CDGHopf}_\k$ if the functor $D$ happens to factor through the category of simplicial abelian groups.

\begin{definition}
	A diagram $D:\Cc\to \sset_*$ is \emph{formal in the category $\Dd$} if there exists a zig-zag of quasi-isomorphisms of $\Dd$, natural with respect to $c$ in $\Cc$, between $C_*(D(c);\k)$ and $H_*(D(c);\k)$.
\end{definition}

\begin{remark}
In this article, we work with diagrams of pointed simplicial sets. One can alternatively consider pointed topological spaces, and replace $\sset_*$ by $\top_*$ in the previous definition. But then a diagram $D\colon\Cc\to \top_*$ is formal if and only if the diagram $\mathrm{Sing}\circ D\colon\Cc\to \sset_*$ is formal. Therefore all the results of the article apply to diagrams of pointed topological spaces as well.      
\end{remark}

\subsection{Models of Eilenberg-MacLane spaces via bar constructions} 

Our favorite model for EML spaces will not be the $K(A[n])$ appearing in proposition~\ref{pr:zigzagofdiagramsofEML}, but rather iterated bar constructions, which allow inductive arguments. We will take~\cite{gu-ma-74} as a reference for bar constructions.
We denote by $\overline{B}A$ the \emph{reduced normalized bar construction} of a graded commutative dg $\k$-algebra $A$ with augmentation $\epsilon\colon A\to \k$. Some explicit formulas for $\overline{B}$ will be recalled later, in sections~\ref{sec-formality} and~\ref{sec-nonformalityFp}. The reduced normalized bar construction yields a functor 
\[
	\overline{B}\colon\mathrm{ACDGAlg}_\k\to \mathrm{CDGHopf}_\k
\]
and since Hopf algebras are augmented algebras, this functor can be iterated.
If $G$ is an abelian group, we may consider the group $\k$-algebra $\k[G]$ as a graded commutative dg Hopf $\k$-algebra concentrated in degree zero, and apply $n$ times the reduced bar construction, to obtain a graded commutative dg Hopf $\k$-algebra $\overline{B}^n \k[G]$. The relevance of bar constructions for our formality problems comes from the following proposition and its corollary. 

\begin{proposition}\label{pr:zigzagforEML}
	Let $\k$ be a commutative ring. For all abelian EML spaces $E$ of height $n$, $C_*(E;\k)$ is connected to $\overline{B}^n \k[\pi_nE]$ by a zig-zag of quasi-isomorphisms of graded commutative dg-$\k$-Hopf algebras, natural with respect to the abelian EML space $E$.
\end{proposition}
\begin{proof}
	The result is essentially proved in~\cite[Appendix~A]{gu-ma-74}. Namely, if $G$ is an abelian group, viewed as a constant simplicial object, the iterated $\overline{W}$-construction  yields an abelian EML space $\overline{W}^nG$ of height $n$ with $n$-th homotopy group naturally isomorphic to $G$ by~\cite[Proposition~A.21]{gu-ma-74}. Then lemma~\ref{lm:2} yields a zig-zag of quasi-isomorphisms of Hopf algebras between $C_*(E;\k)$ and $C_*(\overline{W}^n \pi_n E;\k)$, with $C_*(K(\pi_nE[n]);\k)$ in the middle of the zigzag.
	Thus, it suffices to show that there is a zig-zag of quasi-isomorphisms of Hopf algebras between $\overline{B}^n \k[\pi_nE]$ and $C_*(\overline{W}^n \pi_n E;\k)$, natural with respect to $\pi_nE$. 
	
	Write $\pi:=\pi_nE$ for short. There is an isomorphism $\overline{B}\k[\pi]\simeq C_*(\overline{W}\pi;\k)$ natural with respect to $\pi$ by~\cite[Lemma~A.15]{gu-ma-74}, which finishes the proof 
	for $n=1$. We prove the general case by induction on $n$. Assume that the results holds for $n-1$, hence there is a zig-zag of quasi-isomorphisms of Hopf algebras between $\overline{B}^{n-1}\k[\pi]$ and $C_*(\overline{W}^{n-1}\pi;\k)$, natural with respect to $\pi$. The bar construction $\overline{B}$ preserves quasi-isomorphisms of Hopf algebras, hence by applying it to the latter zig-zag we obtain a new zig-zag of quasi-isomorphisms of Hopf algebras between $\overline{B}^{n}\k[\pi]$ and $\overline{B}C_*(\overline{W}^{n-1}\pi;\k)$. Thus, to finish the proof, it suffices to construct a quasi-isomorphism of Hopf algebras 
	\[
    \overline{\lambda}\colon \overline{B}C_*(\overline{W}^{n-1}\pi;\k)\to C_*(\overline{W}^{n}\pi;\k)\;,
    \]
	natural with respect to $\pi$. 
	Such a quasi-isomorphism is provided by~\cite[Theorem~A.20]{gu-ma-74} by substituting the simplicial Hopf $\k$-algebra $\k[\overline{W}^{n-1}\pi]$ into $A$, in the reference's notation. This finishes the proof for all $n$.
	
	Note that~\cite[Theorem~A.20]{gu-ma-74} does not explicitly assert that $\overline{\lambda}$ is natural with respect to $\pi$, but it is actually the case, and we will now briefly explain why. Let us change the notations and let us denote now by $\overline{W}A$ the normalized chains of the $\overline{W}$ construction of a simplicial augmented algebra $A$ as in~\cite[Definition~A.14]{gu-ma-74}. Let $BC_{\mathrm{norm}}A$ and $WA$ denote the unreduced versions of $\overline{B}C_{\mathrm{norm}}A$ and $\overline{W}A$. Thus $\overline{B}C_{\mathrm{norm}}A$ and $\overline{W}A$ are graded $\k$-submodules of $BC_{\mathrm{norm}}A$ and $WA$, and the latter are actually free right $C_{\mathrm{norm}}A$-modules over the former, and also augmented over $\k$. The statement of~\cite[Theorem~A.20]{gu-ma-74} asserts that there is a unique morphism of augmented right $C_{\mathrm{norm}}A$-modules 
	$$\lambda: BC_{\mathrm{norm}}A\to WA$$
	which preserves the reduced submodules, and $\overline{\lambda}$ is induced by $\lambda$. In order to justify that $\overline{\lambda}$ is natural with respect to $A$, it suffices to prove that $\lambda$ is natural with respect to $A$. Let $s$ denote the contracting homotopy of $WA$ and let $d$ denote the differential of $BC_{\mathrm{norm}}A$. Given a morphism of simplicial commutative Hopf algebras $f:A_1\to A_2$, there is at most one morphism of augmented right $C_{\mathrm{norm}}A_1$-modules $\lambda_{12}:BC_{\mathrm{norm}}A_1\to WA_2$ whose restriction to $\overline{B}C_{\mathrm{norm}}A$ satisfies $s\lambda_{12}d=\lambda_{12}$. Now by the construction of $\lambda$ given in the first line of the proof of~\cite[Theorem~A.20]{gu-ma-74} the compositions $(Wf)\circ \lambda$ and 
	$\lambda\circ BC_\mathrm{norm}f$ both satisfy these conditions, hence they are equal, which proves the naturality of $\lambda$.
\end{proof}

\begin{corollary}\label{cor:reductiontobar}
	Let $\Dd=\mathrm{DGMod}_\k$, $\mathrm{ACDGAlg}_\k$ or $\mathrm{CDGHopf}_\k$.
	A diagram of abelian EML spaces $D\colon\Cc\to \sset_{*}$ of height $n\geq 1$ is formal in $\Dd$ is and only if the following functor is formal:
	\[
	\begin{array}{ccc}
		\Cc & \to & \Dd \\
		c & \mapsto & \overline{B}^n \k[\pi_nD(c)]
	\end{array} \;.
	\]
\end{corollary}

\section{Formality results for diagrams of EML spaces}\label{sec-formality}

The main result of this section is the following theorem.

\begin{theorem}\label{thm:formQ}
	Let $D\colon\Cc\to \sset_*$ be a diagram of abelian EML spaces of a fixed height $n\ge 1$. Let $\k$ be a $\Q$-algebra.
	Then $D$ is formal in $\mathrm{ACDGAlg}_\k$ if $n=1$, and it is formal in $\mathrm{CDGHopf}_\k$ if $n\ge 2$. 
\end{theorem}

For low values of $n$, one can relax the hypothesis that $\k$ is a $\Q$-algebra. To be more specific, if $\k$ is an arbitrary commutative ring, we say that an abelian group $G$ is \emph{$\k$-torsion free} if $\Tor_1^\Z(\k,G)=0$, or equivalently if the order of every torsion element $g\in G$ is invertible in $\k$. We denote by $\Lambda_\k(V)$ the exterior algebra of a $\k$-module $V$, and we define the \emph{$\k$-dimension of $G$} by: 
\[
	\dim(G,\k):=\sup\{d \;|\; \Lambda_\k^d(G\otimes_\Z\k)\ne 0\} = \sup\{d \;|\; \Lambda_\Z^d(G)\otimes_\Z\k \ne 0\} \in \N\cup\{+\infty\}\;.
\] 
Then we have the following enhancement of theorem~\ref{thm:formQ} for $n\le 2$.

\begin{theorem}\label{thm:formsmalldim}
	Let $D:\Cc\to \sset_*$ be a diagram of abelian EML spaces of height $n$. Assume that for all objects $c$ of the indexing category the next two conditions are satisfied:
	\begin{enumerate}
		\item[(1)] the abelian group $\pi_nD(c)$ is $\k$-torsion free,
		\item[(2)] every positive integer less or equal to $\dim(\pi_nD(c),\k)$ is invertible in $\k$.
	\end{enumerate}
	Then $D$ is formal in $\mathrm{ACDGAlg}_\k$ if $n=1$, and it is formal in $\mathrm{CDGHopf}_\k$ if $n= 2$.   
\end{theorem}

\begin{example}\label{ex:kfact}
	If $\k=\Z[1/k!]$ then every abelian group is $\k$-torsion free. In this case, theorem shows that every diagram $D:\Cc\to \sset_*$ of abelian EML spaces of height $n$ such that $\dim(\pi_nD(c),\k)\le k$ for all $c$ is formal in $\mathrm{ACDGAlg}_\k$ if $n=1$, and it is formal in $\mathrm{CDGHopf}_\k$ if $n= 2$.   
\end{example}

By corollary~\ref{cor:reductiontobar}, in order to prove theorems~\ref{thm:formQ} and~\ref{thm:formsmalldim}, we can restrict our attention to formality of reduced bar constructions. So we start by recalling some explicit formulas for reduced bar constructions, following~\cite{gu-ma-74}.

\subsection{Some explicit formulas for bar constructions}
We first recall the \emph{un-normalized reduced bar construction} $\overline{B}_u A$ of an augmented graded commutative dg-algebra $A$. Let $A[1]$ denote the graded $\k$-module such that $A[1]_{i+1}=A_i$. As a graded $\k$-module 
$$\overline{B}_u A=\bigoplus_{p\ge 0} A[1]^{\otimes p}\;.$$
Given homogeneous elements $x_i\in A$, we denote by $[x_1|\cdots|x_p]$ the tensor product $x_1\otimes\cdots\otimes x_p$ viewed as an element of $A[1]^{\otimes p}$. Thus, $[x_1|\dots|x_p]$ has degree $p+\sum_{1\le i\le p}\deg x_i$. For all $i$, we set $\sigma_i=\deg x_1+\dots+\deg x_i$.
Then the product of $\overline{B}_u A$ is the \emph{shuffle product} $*$ defined by 
\[
[x_1|\cdots|x_p]*[x_{p+1}|\cdots|x_{p+q}]=\sum_{t\in \Sigma_{p,q}} (-1)^{q\sigma_p}\epsilon_t\epsilon'_t [x_{t^{-1}(1)}|\cdots|x_{t^{-1}(p+q)}]\]
where $\Sigma_{p,q}$ denotes the set of $(p,q)$-shuffles, $\epsilon_t$ is the signature of a shuffle, $\epsilon'_t$ is the sign such that $x_1\cdots x_{p+q}$ is equal to $ \epsilon'_t x_{t^{-1}(1)}\cdots x_{t^{-1}(p+q)}$ in the free graded commutative algebra on $x_1,\dots,x_n$. The unit is the element $[\;]=1\in\k=A[1]^{\otimes 0}$. The augmentation $\epsilon$ is uniquely determined by $\epsilon(1)=1$. The coproduct $\Delta$ is the \emph{deconcatenation coproduct} defined by
$$\Delta[x_1|\cdots|x_p]=\sum_{0\le i\le p}(-1)^{(p-i)\sigma_i}[x_1|\cdots|x_i]\otimes [x_{i+1}|\cdots|x_p]\;.$$
Finally, the differential is defined by:
\begin{align*}
	d[x_1|\cdots|x_p]=& 
	\epsilon(x_1)[x_2|\cdots|x_p] + (-1)^{p}\epsilon(x_p)[x_1|\cdots|x_{p-1}]
	+\sum_{i=1}^{p-1}(-1)^i[x_1|\cdots|x_ix_{i+1}|\cdots |x_p]\\
	&+\sum_{i=1}^p (-1)^{p+\sigma_{i-1}}[x_1|\cdots|dx_i|\cdots|x_p].
\end{align*}

\begin{remark}\label{rem:signs}
	Our signs follow the conventions of~\cite{gu-ma-74}, which differ from the usual sign conventions, found for example in~\cite{ml-63}, but which coincide with them if $A$ is concentrated in even degrees. 
\end{remark}

The \emph{reduced normalized bar construction} $\overline{B}A$ is the quotient of $\overline{B}_uA$ by the graded $\k$-submodule generated by the elements $[x_1|\cdots|x_p]$ in which at least one of the $x_i$ is equal to $1_A$. The quotient $\overline{B}_uA\to \overline{B}A$ is a quasi-isomorphism of graded commutative dg-Hopf $\k$-algebras.

\begin{remark}\label{rem:reducednormalized}
	Let $\overline{A}=\ker \epsilon$ be the augmentation ideal of $A$.
	Then $\overline{B}'A=\bigoplus_{p\ge 0} \overline{A}[1]^{\otimes p}$ is actually a sub-dg Hopf algebra of $\overline{B}_uA$. Restriction to $\overline{B}'A$ of the quotient $\overline{B}_uA\to \overline{B}A$ yields a natural isomorphism $\overline{B}'A\simeq \overline{B}A$, hence the inclusion $\overline{B}'A\hookrightarrow \overline{B}_uA$ is a quasi-isomorphism.
	In particular, we can work indifferently with $\overline{B}'A$, $\overline{B}A$ or $\overline{B}_uA$ to study formality.
\end{remark}

Next we recall the explicit computation of the homology of reduced bar constructions of symmetric and exterior algebras. 
Let $V$ be a flat $\k$-module. 
Let $S(V[2i])$ and $\Lambda(V[2i+1])$ denote the symmetric and exterior algebras on $V$, where $V$ is viewed as a graded module concentrated in degree $2i$ and $2i+1$ respectively (to lighten the notation, we drop the index $\k$ on $\Lambda$ and $S$).
The canonical inclusion $\Lambda^p(V[2i+1])\subset V[2i+1]^{\otimes p}\subset S(V[2i])[1]^{\otimes p}$ yields an inclusion of dg-Hopf algebras 
$$i_{\Lambda}:\Lambda(V[2i+1])\to \overline{B}S(V[2i])\;.$$

Similarly, let $\Gamma(V[2i+2])$ be the free divided power $\k$-algebra on $V$ viewed as a graded module concentrated in degree $2i+2$. If $\k$ is a $\Q$-algebra, then $\Gamma(V[2i+2])$ is isomorphic to $S(V[2i+2])$ as a graded Hopf algebra. In general, $\Gamma(V[2i+2])$ is the Hopf subalgebra of the shuffle tensor algebra on $V[2i+2]$ which is equal in degree $p(2i+2)$ to the subspace $\Gamma^p(V)$ of the invariants of $V^{\otimes p}$ under the action of $\Sigma_p$.  Multiplying by $(-1)^p$ the canonical inclusion $\Gamma^p(V[2i+2])\subset V[2i+2]^{\otimes p}\subset \Lambda(V[2i+1])[1]^{\otimes p}$ we obtain an inclusion of dg-Hopf algebras:
$$i_{\Gamma}:\Gamma(V[2i+2])\to \overline{B}\Lambda(V[2i+1])\;.$$

\begin{lemma}\label{lm:formalitySLambda}
	For all nonnegative integers $i$ and for all flat $\k$-modules $V$, the maps $i_\Lambda$ and $i_\Gamma$ are quasi-isomorphisms.
\end{lemma}
\begin{proof}
	The result is well-known, see e.g.~\cite[Chapter~3, Ex.~3.2.5]{lv-12}, but our signs on bar constructions differ from the usual ones for $\overline{B}\Lambda(V[2i+1])$, see remark~\ref{rem:signs}. So we must explain why the result still holds with the sign conventions of~\cite{gu-ma-74} that we follow. Observe first that both the source and the target of $i_{\Gamma}$ preserve filtered colimits of the variable $V$. Lazard's theorem asserts that every flat $\k$-module $V$ is a filtered colimit of free modules of finite rank, hence it suffices to prove that $i_\Gamma$ is a quasi-isomorphism when $V$ is free of finite rank. Naturality and compatibility of $i_{\Gamma}$ with products yields a commutative square for all $U$ and $V$ (in which the vertical arrows are induced by the products)
	\[
	\begin{tikzcd}
		\Gamma(U[2i+2])\otimes \Gamma(V[2i+2])\ar{d}{\simeq}\ar{r}{i_{\Gamma}\otimes i_{\Gamma}}   &\overline{B}\Lambda(U[2i+1])\otimes \overline{B}\Lambda(V[2i+1])\ar{d}\\
		\Gamma((U\oplus V)[2i+2])\ar{r}{i_{\Gamma}}& \overline{B}\Lambda((U\oplus V)[2i+1])
	\end{tikzcd}\;.
	\]
	The vertical morphism on the right is a homotopy equivalence by~\cite[Proposition~A.3]{gu-ma-74} and the definition of the shuffle product. Hence we have reduced ourselves to proving the result when $V$ is free of rank one, which is an elementary computation.
\end{proof}
	
\subsection{Proof of theorems~\ref{thm:formQ} and~\ref{thm:formsmalldim}}
Let $\k$ be a commutative ring. By corollary~\ref{cor:reductiontobar}, in order to establish theorems~\ref{thm:formQ} and~\ref{thm:formsmalldim}, it suffices to show that the iterated bar construction $\overline{B}^n\k[G]$ viewed as a functor of the abelian group $G$, is formal.
This will be done in propositions~\ref{prop:level1formalityQ},~\ref{prop:level2formalityQ} and~\ref{prop:higherlevelformalityQ} below.

For all abelian groups $G$, let $G\otimes_\Z\k[1]$ denote a copy of the $\k$-module $G\otimes_\Z\k$ placed in homological degree $1$. For all $\k$-torsion free abelian groups, there is a natural isomorphism of graded $\k$-algebras~\cite[V~Theorem~6.4~(ii)]{br-94}  (to lighten notations, we drop the index $\k$ on exterior algebras):
\[\psi:\Lambda(G\otimes_\Z\k[1])\xrightarrow[]{\simeq} H_*(\overline{B}\k [G])\;.\]
For all $d\in \mathbb{N}\cup\{+\infty\}$ we denote by 
$\Ab_{\le d,\k}$ the full subcategory of abelian groups whose objects are the abelian groups $G$ satisfying $\Tor_1^\Z(\k,G)=0$ and $\dim(G,\k)\le d$. In particular if $\k$ is a flat $\Z$-module then $\Ab_{\le +\infty,\k}=\Ab$. The following proposition proves the case $n=1$ of theorems~\ref{thm:formQ} and~\ref{thm:formsmalldim}.

\begin{proposition}
	\label{prop:level1formalityQ}
	Assume that all the positive integers less or equal to $d$ are invertible in $\k$. 
	Consider $\Lambda(G\otimes_\Z\k[1])$ as an augmented dg $\k$-algebra with zero differential. 
	Then for all $G$ in $\Ab_{\le d,\k}$ there is a quasi-isomorphism $e:\overline{B}\k[G]\to \Lambda(G\otimes_\Z\k[1])$ in $\mathrm{ACDGAlg}_\k$, natural with respect to $G$.
\end{proposition}
\begin{proof}
	By remark~\ref{rem:reducednormalized} it suffices to construct a quasi-isomorphism $e:\overline{B}_u\k[G]\to \Lambda(G\otimes_\Z\k [1])$.
	Let $b_x$, $x\in X$, denote the canonical basis elements of the free $\k$-module $\k[X]$ on a set $X$.
	The object of degree $p$ of $\overline{B}_u\k[G]$ is a free $\k$-module with basis the elements $[b_{g_1}|\dots|b_{g_p}]$ with $g_i\in G$. For all integers $p$ we denote by $e$ the $\k$-linear map defined for $p\le \dim(G,\k)$ by:
	\[
	\begin{array}[t]{cccc}
		e:  & \k[G]^{\otimes p} & \to & \Lambda^p(G\otimes_\Z\k [1]) \\
		& [b_{g_1}|\dots|b_{g_p}] & \mapsto & \dfrac{1}{p!}(g_1\otimes 1)\wedge\cdots\wedge (g_p\otimes 1)
	\end{array}\;,
	\]
	and which is zero if $p>\dim(G,\k)$.
	This $\k$-linear map is obviously natural with respect to $G$. Moreover $e$ is a morphism of complexes: a direct computation shows that $ed[b_{g_1}|\dots|b_{g_p}]=0$ (for this computation, note that the product in the $\k$-algebra $\k[G]$ is given by $b_gb_h=b_{g+h}$.) Another direct computation shows that $e$ is a morphism of $\k$-algebras, and since the domain and the codomain of $e$ are equal to $\k$ in degree zero, $e$ automatically preserves the augmentation, hence it is a morphism in $\mathrm{ACDGAlg}_\k$.
	
	It remains to check that $e$ is a quasi-isomorphism. Since the homology algebras of both the source and the target of $e$ are abstractly isomorphic to exterior algebras generated in degree $1$, it suffices to check that $H_1(e):H_1(\overline{B}\k[G])\to G\otimes_\Z\k$ is an isomorphism, i.e. that the sequence
	\[\k[G]\otimes \k[G]\xrightarrow[]{f} \k[G]\xrightarrow[]{e} G\otimes_\Z\k\to 0\,,\]
	where $f[b_g|b_h]=b_h-b_{gh}+b_g$, is exact at $\k[G]$ and $G\otimes_\Z\k$. The latter holds for all abelian groups $G$, by the usual computation of the degree $1$ homology of an abelian group $G$.
\end{proof}

\begin{remark}
	Let $C(R,M)$ denote the Hochschild complex of a $\k$-algebra $R$ with coefficients in a bimodule $M$. The $\k[G]$-bimodule morphism $\epsilon:\k[G]\to \k$ induces a surjective morphism of complexes $C(\k[G],\k[G])\twoheadrightarrow C(\k[G],\k)=\overline{B}(\k[G])$. The map $e$ of proposition~\ref{prop:level1formalityQ} is the passage to the quotient of the retract of the Hochschild-Kostant-Rosenberg map in characteristic zero from~\cite[Corollary~9.4.4]{we-97}. 
\end{remark}

\begin{remark}
	The isomorphism of algebras $H_*(e):H_*(\overline{B}\k[G])\to \Lambda(G\otimes_\Z\k[1])$
	must actually be an isomorphism of Hopf algebras, since there is a unique graded Hopf algebra structure on its target. However, $e$ is not a morphism of coalgebras, in general.
\end{remark}
\begin{remark}
	The isomorphism $\psi$ lifts to a quasi-isomorphism of dg-algebras $\psi\colon\Lambda(G\otimes_{\mathbb{Z}}\k[1])\to \overline{B}\k[G]$. Although this quasi-isomorphism is natural in $G$ \emph{after taking homology}, it cannot be natural in $G$. Indeed, if it was, then $\psi_{1}\colon G\otimes_{\Z}\k\to \overline{\k[G]}$ would be an injective natural transformation. Now the reduced coproduct of the Hopf algebra $\k[G]$ yields an injective natural transformation $\overline{\Delta}\colon\overline{\k[G]}\to \overline{\k[G]}^{\otimes 2}$, hence $\overline{\Delta}\circ\psi_{1}\colon G\otimes_{\Z}\k\to \overline{\k[G]}^{\otimes 2}$ would be a nonzero natural transformation. Such a nonzero natural transformation cannot exist by Pirashvili's vanishing lemma, see lemma~\ref{lm:pira} below.
\end{remark}
The next proposition proves the case $n=2$ of theorems~\ref{thm:formQ} and~\ref{thm:formsmalldim}.
\begin{proposition}
	\label{prop:level2formalityQ}
	Assume that all the positive integers less or equal to $d$ are invertible in $\k$. 
	Then the functor $\overline{B}^2\k[-]:\Ab_{\le d,\k}\to \mathrm{CDGHopf}_\k$ is formal.
\end{proposition}
\begin{proof}
	If $G$ is a $\k$-torsion free abelian group, then $G\otimes_\Z\k$ is $\k$-flat. (Indeed, if $T$ denotes the torsion subgroup of $G$, then $T\otimes_\Z\k=0$ because all elements of $T$ have order invertible in $\k$, hence $G\otimes_\Z\k\simeq G/T\otimes_\Z\k$ is the $\k$-extension of a flat $\Z$-module, hence it is flat.) 
	Since bar constructions preserve quasi-isomorphisms, proposition~\ref{prop:level1formalityQ} and lemma~\ref{lm:formalitySLambda} yield a zig-zag of natural quasi-isomorphisms of Hopf algebras: 
	\[
    \begin{tikzcd}
    \overline{B}^2\k[G] \arrow{r}{\overline{B}(e)} &
    \overline{B}\Lambda(G\otimes_\Z\k[1]) &
    \Gamma(G\otimes_\Z\k[2]) \arrow[swap]{l}{i_\Gamma}.
    \end{tikzcd}
    \]
\end{proof}

The next proposition finishes the proof of theorem~\ref{thm:formQ}.

\begin{proposition}\label{prop:higherlevelformalityQ}
	Assume that $\k$ is a $\Q$-algebra. For all $n\ge 2$, the functor $\overline{B}^n\k[-]:\Ab\to \mathrm{CDGHopf}_\k$ is formal.   
\end{proposition}
\begin{proof}
	The case $n=2$ is already proved in proposition~\ref{prop:level2formalityQ}. In order to prove the case of $n>2$, one uses that for all $i$, the divided power algebra $\Gamma(G\otimes_\Z\k[2i])$ is naturally isomorphic to the symmetric algebra $S(G\otimes_\Z\k[2i])$ because $\k$ is a $\Q$-algebra. Therefore, by multiple uses of lemma~\ref{lm:formalitySLambda} and of the fact that bar constructions preserve quasi-isomorphisms, we obtain that $\overline{B}^n\k[G]$ is naturally quasi-isomorphic to $\Gamma(G\otimes_\Z\k[2i])$ if $n=2i$ and to $\Lambda(G\otimes_\Z\k[2i+1])$ if $n=2i+1$.
\end{proof}
	
\section{A non-formality result}\label{sec-nonformalityFp}

If $C$ is a cyclic group, we denote by $\langle C\rangle$ the full subcategory of the category of abelian groups $\Ab$, whose objects are all finite direct sums of copies of $C$.
The main result of this section is the following theorem. 
\begin{theorem}\label{thm:nonformality1}
	Assume that $\k$ is not a $\Q$-algebra. Let $C$ be a cyclic group such that $C\otimes_{\Z}\k\ne 0$ and let $n$ be a positive integer. 
    If $D:\langle C\rangle \to \sset_{*}$ is a diagram of EML spaces of height $n$ such that $\pi_nD(G)$ is naturally isomorphic to $G$, then $D$ is not formal in $\mathrm{DGMod}_\k$.    
\end{theorem}
\begin{remark}
Theorem~\ref{thm:nonformality1} shows that theorem~\ref{thm:formsmalldim} does not hold if we drop the hypothesis on the $\k$-dimension of the abelian groups $G$: for example, if $\k=\Z[1/k!]$ and if $C$ is a cyclic group of prime cardinal $p>k$, then the category $\langle C\rangle$ contains groups of arbitrary large $\k$-dimension, and the diagram $G\mapsto K(G,n)$ is not formal in $\mathrm{DGMod}_\k$. (Compare with example \ref{ex:kfact}.)   
\end{remark}
By the rigidification property given in proposition~\ref{pr:zigzagofdiagramsofEML}, it suffices to prove theorem~\ref{thm:nonformality1} when $D$ is a diagram of abelian EML spaces. Then by corollary~\ref{cor:reductiontobar}, theorem~\ref{thm:nonformality1} is equivalent to the following statement on iterated bar constructions.

\begin{theorem}\label{thm:nonformality}
	Assume that $\k$ is not a $\Q$-algebra, and let $C$ be a cyclic group such that $C\otimes_{\Z}\k\ne 0$. Then for all positive $n$ the functor $\overline{B}^n\k[-]:\langle C\rangle\to \mathrm{DGMod}_\k$ is not formal.
\end{theorem}
	
In order to prove theorem~\ref{thm:nonformality}, we will find an obstruction to formality  relying on homological algebra computations in functor categories. Thus we start by recalling some basic facts regarding functor categories.

\subsection{Recollections of functor categories}\label{subsec:recollectF}
If $\Cc$ is an essentially small category, we denote by $\mathcal{F}(\Cc,\k)$ the category whose objects are the functors from $\Cc$ to $\k$-modules, and whose morphisms are the natural transformations. 
We take~\cite{mi-72} as a reference for the basic structural facts on this category.
The category $\mathcal{F}(\Cc,\k)$ is abelian and bicomplete. Kernels, cokernels, direct sums or products are computed objectwise, for example the kernel of a natural transformation $f\colon F\to G$ is such that $(\mathrm{Ker} f)(C)=\mathrm{Ker}(f_C\colon F(C)\to G(C)$. There is a tensor product of functors defined from the tensor product of $\k$-modules by $(F\otimes G)(C)=F(C)\otimes_\k G(C)$, and which is a monoidal structure on $\mathcal{F}(\Cc,\k)$. The category $\mathcal{F}(\Cc,\k)$ also has enough injectives and projectives. The functors $P^X\colon C\mapsto \k[\Hm_\Cc(X,C)]$, where $X$ belongs to $\Cc$, yield a projective generator, and the Yoneda lemma yields a natural isomorphism:
\[
	\Hm_{\mathcal{F}(\Cc,\k)}(P^X,F)\simeq F(X)\;.
\]

When $\Cc$ is an additive category, the category $\mathcal{F}(\Cc,\k)$ has nice extra features. First of all, a functor $F$ is called \emph{reduced} if $F(0)=0$. We denote by $\mathcal{F}_{\mathrm{red}}(\Cc,\k)$ the category of reduced functors and natural transformations. Using the fact that $0$ is initial and terminal in $\Cc$, one sees that every functor $F$ has a canonical decomposition $F=\overline{F}\oplus F(0)$, where $\overline{F}$ is a reduced functor and $F(0)$ denotes the constant functor with value $F(0)$. This decomposition yields a direct sum decomposition 
\[
	\mathcal{F}(\Cc,\k)=\mathcal{F}_{\mathrm{red}}(\Cc,\k)\oplus \mathrm{Mod}_\k\;.
\] 
In particular, for all functors $F$ and $G$ we have
\[
	\Ext^*_{\mathcal{F}(\Cc,\k)}(F,G)=\Ext^*_{\mathcal{F}(\Cc,\k)}(\overline{F},\overline{G})\oplus \Ext^*_\k(F(0),G(0))\;.
\]
Another fundamental result that we will need is Pirashvili's vanishing lemma (see e.g.~\cite[\S2.4]{pi-03} where the proof given works without change when $\Cc$ is an arbitrary essentially small category and $\k$ is an arbitrary commutative ring).
\begin{lemma}[{Pirashvili~\cite{pi-03}}]\label{lm:pira}
	Let $A$, $F_1$ and $F_2$ be three objects of $\mathcal{F}(\Cc,\k)$. Assume that $A$ is additive, and that $F_1$ and $F_2$ are reduced. Then
	$\Ext^*_{\mathcal{F}(\Cc,\k)}(A,F_1\otimes F_2)=0=\Ext^*_{\mathcal{F}(\Cc,\k)}(F_1\otimes F_2,A)$.
\end{lemma}

\subsection{Proof of theorem~\ref{thm:nonformality}}
We fix a commutative ring $\k$ which is not a $\Q$-algebra, and a cyclic group $C$ such that $C\otimes_{\mathbb{Z}}\k\ne 0$. 
\begin{lemma}\label{lm:field}
	The ring $\k$ has a quotient field $\F$ of characteristic $p>0$ such that $C\otimes_{\Z}\F\simeq \F$.
\end{lemma}
\begin{proof}
	If $C$ is finite of cardinal $q$, then $q$ is not invertible in $\k$ since $C\otimes_{\Z}\k\ne 0$. Hence there is a prime $p$ dividing this cardinal which is not invertible in $\k$. If $C$ is infinite, the fact that $\k$ is not a $\Q$-algebra guarantees that there is a prime $p$ which is not invertible in $\k$. In both cases, one may take $\F$ to be the quotient of $\k/p\k$ by some maximal ideal. 
\end{proof}

We fix a field $\F$ of characteristic $p$ as in lemma~\ref{lm:field}. 
For instance, if $\k=\Z$ or $\k=\F_{p}[x]$ then $\F=\F_{p}$.
We introduce the following objects of $\mathcal{F}(\langle C\rangle,\k)$. 
\begin{itemize}
	\item We denote by $A$ the nonzero additive functor $A(G):=G\otimes_{\Z}\F$.
	\item We denote by $P$ the functor $P(G):=\k[G]$, and by $\overline{P}$ its reduced part. Thus $P(G)$ is the $\k$-algebra of the group $G$, and $\overline{P}(G)$ is its augmentation ideal. Note that $P$ (and hence its summand $\overline{P}$) is projective since $P\simeq P^{C}=\k[\Hm_{\Z}(C,-)]$. 
	\item We denote by $H_i(n,\k)$ the functor $G\mapsto H_i(\overline{B}^n\k[G])$.
\end{itemize}
Our proof of theorem~\ref{thm:nonformality} will rely on two preliminary $\Ext$-computations. The first one is a fairly easy consequence of the basic facts recalled in section~\ref{subsec:recollectF}.
\begin{proposition}\label{prop:Extcomputation1}
	Let $i$ and $k$ be two nonnegative integers. We have 
	\[
		\Ext^i_{\mathcal{F}(\langle C\rangle,\k)}(\overline{P}^{\otimes k},A)=
		\begin{cases}
			\F & \text{if $k=1$ and $i=0$,}\\
			0  & \text{otherwise.}
		\end{cases}
	\]
\end{proposition}
\begin{proof}
	If $k=0$ then $\overline{P}^{\otimes 0}$ is the constant functor with value $\k$ and $A$ is reduced, hence there is no nonzero $\Ext$ between them.
	If $k=1$ the $\Ext$ of positive degree vanish because $\overline{P}$ is projective, and $\Ext^0_{\mathcal{F}(\langle C\rangle,\k)}(\overline{P},A)=A(C)\simeq \F$ by the Yoneda lemma and lemma~\ref{lm:field}. Finally, the $\Ext$ vanishing for $k>1$ comes from Pirashvili's vanishing lemma~\ref{lm:pira}. 
\end{proof}

\begin{proposition}\label{prop:Extcomputation2}
	For all positive $n$, the $\F$-vector space $\bigoplus_{i,j\ge 0}\Ext^i_{\mathcal{F}(\langle C\rangle,\k)}(H_j(n,\k),A)$ has dimension greater than one.
\end{proposition}

The proof of proposition~\ref{prop:Extcomputation2} is significantly more technical than that of proposition~\ref{prop:Extcomputation1}. The reason is that no workable description of the functors $H_i(n,\k)$ is known when $\k$ is an arbitrary commutative ring. We postpone the proof of proposition~\ref{prop:Extcomputation2} to sections~\ref{subsec:proof1Z},~\ref{subsec:proof1ZkZ} and~\ref{subsec:proofn}.
We first finish the proof of theorem~\ref{thm:nonformality} assuming that proposition~\ref{prop:Extcomputation2} holds.
\begin{proof}[Proof of theorem~\ref{thm:nonformality}]
	We proceed by contradiction, so we assume that $\overline{B}^n\k[-]$ is formal. If we regard $\overline{B}^n\k[-]$ as a complex in $\mathcal{F}(\langle C\rangle,\k)$, then formality means that it is connected to the complex $H(n,\k)=\bigoplus_{i\ge 0}H_i(n,\k)$ (with zero differential) by a zig-zag of quasi-isomorphisms of complexes in $\mathcal{F}(\langle C\rangle,\k)$. Thus, if $J^A$ denotes an injective resolution of $A$ in $\mathcal{F}(\langle C\rangle,\k)$, then the totalization of the bicomplex of $\k$-modules
	\[
		T:=\Hm_{\mathcal{F}(\langle C\rangle,\k)}(\overline{B}^n\k[-],J^A)
	\]
	is connected by a zig-zag of quasi-isomorphisms in $\mathrm{DGMod}_\k$ to the total complex of $\k$-modules
	\[
		T':=\Hm_{\mathcal{F}(\langle C\rangle,\k)}(H(n,\k),J^A)\;.
	\]
	
	Let us compute the homology of $T$. Let $\overline{B}^n\k[-]_s$ denote the object of degree $s$ of the complex  $\overline{B}^n\k[-]$. There is a first quadrant spectral sequence:
	\[
		E_1^{st}=\Ext^t_{\mathcal{F}(\langle C\rangle,\k)}(\overline{B}^n\k[-]_s,A)\Rightarrow H^{s+t}(T)\;.
	\]
	As a graded object, the reduced normalized bar complex of an augmented dg-algebra $A$ is isomorphic to $\bigoplus_{k\ge 0}\overline{A}[1]^{\otimes k}$ with the convention that $\overline{A}[1]^{\otimes 0}:=\k$, 
	thus we see by induction on $n$ that:
	\[ 
		\overline{B}^n\k[-]_s = 
		\begin{cases}
			\k & \text{if $s=0$,}\\
			0 & \text{if $0<s<n$,}\\
			\overline{P} & \text{if $s=n$,}
		\end{cases}
	\]
	and for $s>n$, $\overline{B}^n\k[-]_s$ is a direct sum of copies of functors $\overline{P}^{\otimes k}$ with $k>1$. Therefore by proposition~\ref{prop:Extcomputation1}, we have $E_1^{st}\simeq\F$ if $(s,t)=(n,0)$ and zero otherwise. As a consequence, the spectral sequence collapses and the total homology of $T$ is isomorphic to $\F$.
	
	Now the total homology of $T'$ is $\bigoplus_{i,j\ge 0}\Ext^i_{\mathcal{F}(\langle C\rangle,\k)}(H_j(n,\k),A)$, which is an $\F$-vector space of dimension greater than one by proposition~\ref{prop:Extcomputation2}. This contradicts the fact that $T$ and $T'$ have isomorphic homology. Hence $\overline{B}^n\k[-]$ cannot be formal.
\end{proof}

\subsection{Proof of proposition~\ref{prop:Extcomputation2} when $C$ is infinite}\label{subsec:proof1Z}
We first recall a (very) partial description of the homology of iterated bar constructions, contained in~\cite{ca-54}.
For all free abelian groups $G$, let us denote by $U(G[n])$ the free divided power ring on a copy of $G$ viewed as a $\Z$-module placed in degree $n$. Thus $U(G[n])=\Lambda(G[n])$ if $n$ is odd and $U(G[n])=\Gamma(G[n])$ if $n$ is even. 
\begin{lemma}\label{lm:clecar0}
	There is a decomposition, natural with respect to the free abelian group $G$
	\[
		H_*(\overline{B}^n\Z[G])=U(G[n])\oplus T(G,n),
	\]
	where $T(G,n)$ is a functor with values in torsion abelian groups. 
	As a consequence, for all commutative rings $\k$, the graded $\k$-module $U(G[n])\otimes_\Z\k$ is a natural direct summand of $H_*(\overline{B}^n\k[G])$.
\end{lemma}
\begin{proof}
	Assume first that $\k=\Z$. Let $G$ denote a free abelian group and
	let $K(G,n)$ denote an Eilenberg-MacLane space such that $\pi_nK(G,n)\simeq G$ naturally with respect to $G$. The Hurewicz theorem and proposition~\ref{pr:zigzagforEML} yield a graded isomorphism $G[n]\simeq H_n(K(G,n);\Z)\simeq H_n(\overline{B}^n\Z[G])$. This induces an algebra morphism $U(G[n])\to H_n(\overline{B}^n\Z[G])$, natural with respect to $G$. Then~\cite[Exp.~11,~Thm.~1]{ca-54} shows that this algebra morphism is injective and in direct sum with the torsion part of the homology. 
	The statement for an arbitrary ground ring $\k$ now follows from the universal coefficient theorem.
\end{proof}
	
By lemma~\ref{lm:clecar0}, $H(\overline{B}^n\k[G])$ admits $\bigoplus_{d\ge 0}\Gamma^d(G)\otimes_\Z\k$ as a natural direct summand if $n$ is even and $\bigoplus_{d\ge 0}\Lambda^d(G)\otimes_\Z\k$  if $n$ is odd (we don't specify the homological degrees here, which are not relevant to prove proposition~\ref{prop:Extcomputation2}). Hence, proposition~\ref{prop:Extcomputation2} follows from the next lemma.
\begin{lemma}
	Let $\k$ be a commutative ring having a quotient field $\F$ of positive characteristic $p$, let $\langle C\rangle$ denote the category of free abelian groups of finite rank, and let $A:\langle C\rangle\to \mathrm{Mod}_\k$ denote the functor $G\mapsto G/pG\otimes_\Z\k$. 
	If $F\colon\langle C\rangle\to \mathrm{Mod}_\k$ is the functor $F(G)=\Lambda^1(G)\otimes_\Z\k=\Gamma^1(G)\otimes_\Z\k=G\otimes_\Z\k$ or $F(G)=\Gamma^p(G)\otimes_\Z\k$, then
	\[
		\Ext^0_{\mathcal{F}(\langle C\rangle,\k)}(F,A)\ne 0\;.
	\]
	Furthermore, if $F(G)=\Lambda^p(G)\otimes_\Z\k$ then 
	\[
		\Ext^{p-1}_{\mathcal{F}(\langle C\rangle,\k)}(F,A)\ne 0\;.
	\]
\end{lemma}
\begin{proof}
	Assume first that $F(G)=G\otimes_Z\k$. Then tensoring the quotient map $\k\to\F$ by $G$ gives a non-zero element in $\Ext^0_{\mathcal{F}(\langle C\rangle,\k)}(F,A)$. 
	Assume now that $F(G)=\Gamma^p(G)\otimes_\Z\k$. Then the quotient map $\k\to\F$ together with the Verschiebung map $v\colon\Gamma^p(G)\otimes_\Z\F_p=\Gamma_{\F_p}(G/pG)\to G/pG$ induce a surjective natural transformation
	\[
		F(G)=\Gamma^p(G)\otimes_\Z\k\to 	\Gamma^p(G)\otimes_\Z\F=\Gamma^p(G)\otimes_\Z\F_p\otimes_{\F_p}\F\xrightarrow[]{v\otimes\F} G/pG\otimes_{\F_p}\F=A(G),
	\]
	hence a nonzero element in $\Ext^0_{\mathcal{F}(\langle C\rangle,\k)}(F,A)$. Finally, assume that $F(G)=\Lambda^p(G)\otimes_\Z\k$. 
	Let $K^i(G)$ denote the tensor product $(\Gamma^{p-i}(G)\otimes_\Z\k)\otimes_\k (\Lambda^i(G)\otimes_\Z\k)$, in particular $K^p(G)\otimes_\Z\k=F(G)$ and $K^0(G)=\Gamma^p(G)\otimes_\Z\k$.
	There is an exact Koszul complex, see e.g.~\cite[sec.~6.2.1]{br-mi-to-16}:
	\[
		0\to K^0(G)\to K^1(G)\to \cdots\to K^p(G)\to 0\;.
	\] 
	Each $K^i(G)$ is a tensor product of two reduced functors of $G$ if $0<i<p$, hence by Pirashvili's vanishing lemma the functors $K^i$ are acyclic for $\Ext^0_{\mathcal{F}(\langle C\rangle,\k)}(-,A)$. Moreover, we have already shown above that $\Ext^0_{\mathcal{F}(\langle C\rangle,\k)}(K^0,A)\ne 0$. Thus dimension shifting (see e.g.~\cite[Exercise~2.4.3,~p.47]{we-97} or~\cite[III~Section~7]{br-94}) gives an isomorphism
	\[
		\Ext^{p-1}_{\mathcal{F}(\langle 	C\rangle,\k)}(F,A)=\Ext^{p-1}_{\mathcal{F}(\langle C\rangle,\k)}(K^p,A)\simeq \Ext^0_{\mathcal{F}(\langle C\rangle,\k)}(K^0,A)\ne 0\;.
	\]
\end{proof}

\subsection{Proof of proposition~\ref{prop:Extcomputation2} when $n=1$ and $C$ is finite}\label{subsec:proof1ZkZ}
Let us assume that $C=\Z/q\Z$. 
\begin{lemma}\label{lm:tenstor}
	Let $A$ be an abelian group and let $_qA$ denote its subgroup of $q$-torsion.
	For all free $\Z/q\Z$-modules $G$, there is an isomorphism $\Tor^{\Z}_1(G,A)\simeq G\otimes_\Z({}_qA)$ natural with respect to $G$.
\end{lemma}
\begin{proof}
	The result follows from the fact that every functor $F$ from free $\Z/q\Z$-modules to $\Ab$ preserving all direct sums is completely determined by the abelian group $F(\Z/q\Z)$, see e.g.~\cite{ei-60}.    
\end{proof}
The next lemma  relies on the known explicit functorial description of the homology of $H_*(\overline{B}\k[G])$ when $\k$ is a field.
\begin{lemma}\label{lm:pfn1Fp}
	Let $p$ be a prime dividing $q$. For all free $\Z/q\Z$-modules $G$, there is a non-zero morphism $H_{2p}(\overline{B}\F_p[G])\to G/pG$, natural with respect to $G$.
\end{lemma}
\begin{proof}
	There is a natural surjective morphism $H_{2p}(\overline{B}\F_p[G])\to \Gamma^p({}_pG)$, where the target is the $p$-homogeneous part of the free $\F_p$-divided power algebra on the $p$-torsion subgroup $_pG$, see~\cite[V~Theorem~(6.6)]{br-94} in odd characteristic and~\cite[Corollary~4.2]{iv-za-19} in characteristic 2. Composing this map with the Verschiebung map $\Gamma^p({}_p G)\to {}_pG$ and the natural isomorphism $_pG\simeq G/pG$ of lemma~\ref{lm:tenstor}, we obtain a natural surjective map $H_{2p}(\overline{B}\F_p[G])\to  G/pG$.    
\end{proof}

In order to extend the result of lemma~\ref{lm:pfn1Fp} to $\k=\Z$ by the universal coefficient theorem, we need the following elementary non-functorial information on the cohomology of $H_*(\overline{B}\Z[G])$.
\begin{lemma}\label{lm:descriptionHBGZ}
	If $C$ is a cyclic group of cardinal $q$, then for all positive integers $i$ and $r$, $H_i(\overline{B}\Z[C^{\oplus r}])$ is a free $\Z/q\Z$-module.
\end{lemma}
\begin{proof}
	The result is known for $r=1$ by an explicit computation~\cite[II~(3.1)]{br-94}, and the result for arbitrary $r$ follows by the K\"unneth theorem.
\end{proof}
\begin{lemma}\label{lm:pfn1Z}
	Let $p$ be a prime dividing $q$. For all free $\Z/q\Z$-modules $G$, there is a non-zero morphism $H_{2p-1}(\overline{B}\Z[G])\to G/pG$ natural with respect to $G$.
\end{lemma}
\begin{proof}
	For all $i>0$, lemmas~\ref{lm:descriptionHBGZ} and~\ref{lm:tenstor} yield an isomorphism $_p H_i(\overline{B}\Z[G])\simeq H_i(\overline{B}\Z[G])\otimes_\Z\F_p\;$ natural with respect to the free $\Z/q\Z$-module $G$.
	Thus, the universal coefficient theorem yields a short exact sequence of functors
	\[ 
		0\to 	H_{2p}(\overline{B}\Z[G])\otimes_\Z\F_p\to H_{2p}(\overline{B}\F_p[G])\to H_{2p-1}(\overline{B}\Z[G])\otimes_\Z\F_p\to0.
	\]
	Lemma~\ref{lm:pfn1Fp} yields a natural map $f\colon H_{2p}(\overline{B}\F_p[G])\to G/pG$. Let $f'$ be the restriction of $f$ to $H_{2p}(\overline{B}\Z[G])\otimes_\Z\F_p$. The functor $G\mapsto G/pG$ is simple (indeed, if $x$ is a non-zero element of $G/pG$ then for all $y\in G/pG$, there is a linear map $u\colon G\to G$ such that $u(x)=y$, therefore if $x$ belongs to a  subfunctor of $G/pG$ then every $y$ also belongs to this subfunctor, hence the subfunctor is equal to $G/pG$). Hence $f'$ must be either surjective for all $G$, or zero. But $H_{2p}(\overline{B}\Z[\Z/q\Z])=0$ by~\cite[II~(3.1)]{br-94} hence $f'=0$. Thus $f$ factors through $H_{2p-1}(\overline{B}\Z[G])\otimes_\Z\F_p$, whence the result. 
\end{proof}

The next lemma proves proposition~\ref{prop:Extcomputation2} for $n=1$ and $\k$ arbitrary.
\begin{lemma}
	Let $C=\Z/q\Z$, let $p$ be a prime dividing $q$ and let $\k$ be a commutative ring having a quotient field $\F$ of characteristic $p$. 
	The $\F$-vector space $\bigoplus_{j\ge 0}\Ext^0_{\mathcal{F}(\langle C\rangle,\k)}(H_j(1,\k),A)$ has dimension greater than one.
\end{lemma}
\begin{proof}
	The quotient map $\k\to\F$ induces a surjective map $H_1(\overline{B}\k[G])=G\otimes_\Z\k\to G\otimes_\Z\F=A(G)$, natural with respect to $G$. Hence $\Ext^0_{\mathcal{F}(\langle C\rangle,\k)}(H_1(1,\k),A)$ is non-zero. 
	
	Thus, it remains to exhibit a non-zero morphism $H_j(1,\k)\to A$ for some $j\ne 1$.
	Let 
	\[
		T:=\Tor^\Z_1(\Z/q\Z,\k). 
	\]
	If $G$ is a free $\Z/q\Z$-module, lemmas~\ref{lm:descriptionHBGZ} and~\ref{lm:tenstor} yield an isomorphism 
	\[
		H_{j-1}(\overline{B}\Z[G])\otimes_\Z T\simeq \Tor^\Z_1(H_{j-1}(\overline{B}\Z[G]),\k).
	\]
	Thus for all positive integers $i$, the universal coefficient theorem yields a short exact sequence, natural with respect to $G$:
	\begin{equation}\label{eq:uctt}
		0\to H_j(\overline{B}\Z[G])\otimes_\Z\k\to H_j(\overline{B}\k[G])\to H_{j-1}(\overline{B}\Z[G])\otimes_\Z T\to 0\;.
	\end{equation}
	
	Assume first that $T=0$, then we claim that there is a non-zero morphism $H_{2p-1}(1,\k)\to A$. Indeed $H_{2p-1}(\overline{B}\k[G])$ is isomorphic to $H_{2p-1}(\overline{B}\Z[G])\otimes_\Z\k$ by~\eqref{eq:uctt}, hence it surjects onto $H_{2p-1}(\overline{B}\Z[G])\otimes_\Z\F$, and by lemma~\ref{lm:pfn1Z} there is a non-zero natural map from the latter to $G/pG\otimes_\Z \F =A(G)$.
	
	Assume now that $T\ne 0$, then we claim that there is a nonzero morphism $H_{2p}(1,\k)\to A$. Indeed, $T$ is a non-zero $\Z/q\Z$-module, hence there is a morphism of abelian groups $T\to \F_p$. Together with the quotient map of the exact sequence~\eqref{eq:uctt}, this morphism yields a surjective natural map:
	$H_j(\overline{B}\k[G])\to H_{j-1}(\overline{B}\Z[G])\otimes_\Z\F_p$. And lemma~\ref{lm:pfn1Z} provides a nonzero natural map
	\[
		H_{2p-1}(\overline{B}\Z[G])\otimes_\Z\F_p \to G/pG\hookrightarrow G/pG\otimes_\Z\F=A(G)
	\]
	whence the result.
\end{proof}

\subsection{Proof of proposition~\ref{prop:Extcomputation2} when $n>1$ and $C$ is finite}\label{subsec:proofn} 
We recall~\cite[Section~14]{ei-ml-53},~\cite[Exp.~6]{ca-54} that for all commutative rings $\k$ and for all $n>0$, the \emph{suspension} is a morphism of graded $\k$-modules 
\[
	\sigma\colon H_*(\overline{B}^n\k[G])\to H_{*+1}(\overline{B}^{n+1}\k[G])
\]
natural with respect to $G$, which vanishes on products of classes of positive degrees (a.k.a. \emph{decomposable elements}). Moreover, the suspension induces an isomorphism in degrees $0< *< 2n$, hence the modules $H_{i+n}(\overline{B}^n\k[G])$ do not depend on $n$ if $n>i$, they are called the \emph{stable homology of $G$} and denoted by $H_i^{\mathrm{st}}(G,\k)$. 
For $\k=\F_p$, the stable homology was described as a functor in~\cite[Exp.~10,~\S6]{ca-54}, in particular for all $i\ge 0$,  $H_i^{\mathrm{st}}(G,\F_p)$ is a finite direct sum of copies of $_pG$ and $G/pG$ (with multiplicities known and independent on $G$, where $_pG$ denotes the $p$-torsion subgroup of $G$).

We also recall~\cite[Exp.~6]{ca-54} that the homology of the bar constructions $\overline{B}^n\F_p[G]$ is also equipped with further homology operations, namely the $p$-th divided power $\gamma_p$ and transpotence $\phi_p$. Suitable words (called  \emph{$p$-admissible words} in~\cite{ca-54}) in the letters $\gamma_p$, $\phi_p$ and $\sigma$ encode compositions of homology operations which yield injective morphisms 
\[
	G/pG\to H_*(\overline{B}^n\F_p[G]) \qquad\text{ and }\qquad  {}_pG\to H_*(\overline{B}^n\F_p[G])
\]
whose images generate $H_*(\overline{B}^n\F_p[G])$ as a free divided power algebra, see the `théorème fondamental' of~\cite[Exp.~9,~\S3]{ca-54} if $p$ is odd and of~\cite[Exp.~10,~\S2]{ca-54} if $p=2$.

\begin{lemma}\label{lm:stablemap}
	Let $q$ be a positive integer and let $p$ be a prime dividing $q$. For all $n\ge 2$, let $\theta_*\colon H_{*}(\overline{B}^n\Z[G])\to H_{*}(\overline{B}^n\F_p[G])$ denote the graded morphism induced by the morphism $\Z\to \F_p$. For all free $\Z/q\Z$-modules $G$, there is a morphism $f\colon H_{2p+n-1}(\overline{B}^n\F_p[G])\to G/pG$, natural with respect to $G$ such that $f\circ \theta_{2p+n-1}\ne 0$.
\end{lemma}
\begin{proof}
	Let $G$ be an arbitrary abelian group.
	Let $\alpha=\phi_2$ if $p$ is odd and $\alpha=\gamma_2\sigma$ if $p=2$. Then for all $N\ge 0$, the word $\sigma^{N+n-1}\gamma_p\alpha$ is $p$-admissible, and it is associated to an injective composition of homology operations:
	\[
		_pG\xrightarrow[]{\alpha}H_2(\overline{B}\F_p[G])\xrightarrow[]{\sigma\gamma_p} H_{2p+1}(\overline{B}^2\F_p[G])\xrightarrow[]{\sigma^{N+n-2}}H_{2p+N+n-1}(\overline{B}^{n+N}\F_p[G]).
	\]
	In particular, if $N\gg 0$, this composition of operations yields an injective morphism from $_pG$ to the stable homology $H^{\mathrm{st}}_{2p-1}(G,\F_p)$.
	
	Now we restrict our attention to free $\Z/q\Z$-modules $G$, hence lemma~\ref{lm:tenstor} yields a natural isomorphism $_pG\simeq G/pG$, thus $H_{2p-1}^{\mathrm{st}}(G,\F_p)$ is naturally isomorphic to $G/pG^{\oplus s}$ for some $s>0$. We let $f\colon H_{2p+n-1}(\overline{B}\F_p[G])\to G/pG$ be a coordinate of the morphism $\sigma^N\colon H_{2p+n-1}(\overline{B}\F_p[G])\to H_{2p-1}^{\mathrm{st}}(G,\F_p)\simeq G/pG^{\oplus s}$ which does not vanish on the image of $\sigma^{n-1}\gamma_p\alpha$ in $H_{2p+n-1}(\overline{B}^{n}\F_p[G])$. 
	
	In order to prove the lemma, it remains to construct an element $u\in H_{2p+n-1}(\overline{B}^n\Z[G])$ such that $f(\theta_{2p+n-1}(u))\ne 0$. Let $x\in {}_pG$ be a nonzero element and let $y=\beta_p\sigma^{n-2}\phi_p\alpha(x)$, where $\beta_p$ denotes the Bockstein operation of $H_{2p+n-1}(\overline{B}^n\F_p[G])$. Since Bocksteins commute with suspensions,~\cite[Exp.~8, Thm.~1]{ca-54} shows that $y-\sigma^{n-1}\gamma_p\alpha(x)$ is decomposable (and it is actually zero if $p$ is odd). Since $f$ vanishes on decomposable elements, we obtain $f(y)=f(\sigma^{n-1}\gamma_p\alpha(x))\ne 0$. 
	Now the bockstein $\beta_p$ factors through $\theta_*$ (see e.g.~\cite[Exp.~8,~p.8-03]{ca-54}), hence $y=\theta_{2p+n-1}(u)$ for some element $u\in H_{2p+n-1}(\overline{B}^n\Z[G])$, which finishes the proof. 
\end{proof}
The next lemma proves proposition~\ref{prop:Extcomputation2} for $n>1$ and $\k$ arbitrary.
\begin{lemma}
	Let $C=\Z/q\Z$, let $p$ be a prime dividing $q$ and let $\k$ be a commutative ring having a quotient field $\F$ of characteristic $p$. For all $n\ge 2$,
	the $\F$-vector space $\bigoplus_{j\ge 0}\Ext^0_{\mathcal{F}(\langle C\rangle,\k)}(H_j(n,\k),A)$ has dimension greater than one.
\end{lemma}
\begin{proof}
	The homology $H_n(\overline{B}^n\k[G])$ is naturally isomorphic to $G\otimes_\Z\k$, hence it has $G\otimes_\Z\F$ as a natural quotient, which proves that there is a non-zero morphism $H_n(n,\k)\to A$. Thus, in order to prove the lemma, it suffices to prove that there is a non-zero morphism $H_{2p+n-1}(n,\k)\to A$. 
	There is a commutative diagram of abelian groups, natural with respect to $G$, in which the unadorned arrows are induced by the ring morphisms $\Z\to \k$, $\k\to \k/p\k$, $\Z\to \F_p$ and $\F_p\to \k/p\k$, $f$ is the map provided by lemma~\ref{lm:stablemap} and $f'$ is the unique map making the right square commutative:
	\[
	\begin{tikzcd}
		H_{2p+n-1}(\overline{B}^n\k[G])\ar{r}{}&H_{2p+n-1}(\overline{B}^n\k/p\k[G])\ar[dashed]{r}{f'}&A(G)=G/pG\otimes_\Z\F\\
		H_{2p+n-1}(\overline{B}^n\Z[G])\ar{r}{}\ar{u}{}&H_{2p+n-1}(\overline{B}^n\F_p[G])\otimes_\Z\k/p\k\ar{u}{\simeq}\ar{r}{f\otimes\k/p\k}&G/pG\otimes_\Z\k/p\k\ar[two heads]{u}
	\end{tikzcd}.
	\]
	By lemma~\ref{lm:stablemap} the composite map from the bottom left corner to the top right corner is non-zero. Therefore the first row of the diagram provides a non-zero morphism $H_{2p+n-1}(n,\k)\to A$.
\end{proof}
	
\section{Applications}\label{sec:app}

We now explain how formality results apply to the study of singular (co)homology of homotopy colimits. 
Recall that if $\Mm$ is a cofibrantly generated model category (such as topological spaces or simplicial sets), the category $\Mm^\Cc$ of diagrams of shape $\Cc$ in $\Mm$ carries a model structure called the \emph{projective model structure}. The homotopy colimit of a diagram $D\colon\Cc\to \Mm$ is then defined as the colimit of a cofibrant replacement of $D$ in this projective model structure. 

Recall furthermore that the simplicial replacement of a diagram $D\colon\Cc\to \Mm$ is the simplicial object in $\Mm$ denoted by $\srep_\bullet D$ which is given in degree $n$ by 
\[
	\srep_nD=\bigsqcup_{\sigma =c_0\ot \cdots\ot c_{n}\in N_{n}(\Cc)} D(c_{n})\;,
\]
where $N_\bullet(\Cc)$ denotes the nerve of $\Cc$. If $s_i'$ and $d_i'$ denote the  degeneracies and the faces of the nerve of $\Cc$, then the degeneracy $s_i\colon\srep_{n-1}D\to \srep_{n}D$ sends the component indexed by $\sigma$ identically to the component indexed by $s_i'(\sigma)$, and the face $d_i\colon\srep_nD\to \srep_{n-1}D$ sends the component indexed by $\sigma=c_0\ot \cdots\ot c_{n}$ to the component indexed by $d_i'(\sigma)$, identically if $i<n$, or acts as $D(c_n\to c_{n-1})$ if $i=n$.
When $\Mm=\top$ or $\sset$, the homotopy colimit of a diagram can be alternatively computed as the geometric realization of its simplicial replacement, namely it follows from~\cite[Theorem~18.7.5~(1)]{hi-03} and~\cite[Corollary~A.6]{du-is-04} that there is a weak equivalence:
\[
	\hoc D \xrightarrow[]{\sim} |\srep_\bullet(D)|\;.
\]

We are interested in the singular homology of the homotopy colimit of a diagram $D$ of spaces (topological spaces or simplicial sets). Let $C_{**}(D)$ denote the first quadrant homological bicomplex of $\k$-modules whose columns are the singular chain complexes $C_{p*}(D) = C_*(\srep_pD;\k)$ and whose horizontal differentials $C_{p\,q}(D)\to C_{(p-1)\,q}(D)$ are induced by the simplicial structure of $\srep_\bullet D$.
\begin{lemma}\label{lm:abut}
	For all diagrams $D$ in $\sset$ or $\top$, there is a zig-zag of homotopy equivalences natural with respect to $D$, connecting the chain complexes $\Tot C_{**}(D)$ and $C_*(\hoc D;\k)$.
\end{lemma}
\begin{proof}
	Assume first that $D$ is a diagram in $\sset$. Then~\cite[Theorem~15.11.6]{hi-03} gives an isomorphism of simplicial sets $|\srep_\bullet D|\simeq \mathrm{diag}\,\srep_\bullet D$, and the singular chains of the latter are naturally homotopy equivalent to $\Tot C_{**}(D)$ by the Eilenberg-Zilber theorem. Assume now that $D$ is a diagram in $\top$. Let $D'$ be the diagram of simplicial sets given by the singular set of $D$, i.e., $D'(c)=\mathrm{Sing}_\bullet(D(c))$. Then $D$ is weakly equivalent to the diagram $|D'|\colon c\mapsto |D'(c)|$ given by geometric realization, hence it suffices to prove the result for $D=|D'|$. Note that if $QD'\to D'$ is a cofibrant replacement of $D'$ in $\sset^\Cc$ then by applying geometric realization one obtains a cofibrant replacement $|QD'|\to |D'|$ in $\top^\Cc$. Since geometric realization preserves colimits, this gives a weak homotopy equivalence
	\[
		\hoc |D'|\simeq |\hoc D'|\;.
	\]
	Therefore, the singular chain complex $C_*(\hoc |D'|;\k)$ is 
	naturally quasi-isomorphic to\\ $C_*(|\hoc D'|;\k)$. Since for all simplicial sets $X$ the adjunction unit $X\to \mathrm{Sing}_\bullet |X|$ is a weak homotopy equivalence, $C_*(|\hoc D'|;\k)$ is naturally quasi-isomorphic to $C_*(\hoc D';\k)$. We have already shown that the latter is naturally 
	quasi-isomorphic to $\Tot C_{**}(D')$. Finally, by using the adjunction unit  $X\to \mathrm{Sing}_\bullet |X|$ again, we obtain that the latter is naturally quasi-isomorphic to $\Tot C_{**}(|D'|)$.
\end{proof}

\begin{remark}\label{rem:upzigzag}
The zigzag in lemma~\ref{lm:abut} lifts to a zigzag of quasi-isomorphisms in the category of strictly coassociative DG coalgebras. 
The standard Alexander-Whitney diagonal on singular chains is coassociative but not strictly cocommutative. 
Therefore, applying the $\Q$-dual to this zigzag yields a statement for singular cochain complexes in the category of strictly associative DG algebras, but not commutative ones. 
\end{remark}

Thus, $C_{**}(D)$ is a first quadrant homological bicomplex of $\k$-modules whose total complex is quasi-isomorphic to $C_*(\hoc D;\k)$. We consider the spectral sequence $(E_{**}^r,d^r)_{r\ge 1}$ associated to its columnwise filtration. 
If $H_d(D)\colon\Cc\to \mathrm{Mod}_\k$ denotes the functor sending each $c$ to the singular homology of $D(c)$, then the first page $(E^1_{**}(D), d^1)$ is given by:
\begin{align*}
	&E^1_{p,q}(D)= \srep_p H_q(D)\;,\\
	&d^1:E^1_{p\,q}\to E^1_{(p-1)\,q} \text{ is induced by the simplicial structure of $\srep_\bullet H_q(D)$}\;.
\end{align*}
Hence the homology of $(E^1_{*q}(D),d^1)$ is precisely the homology $H_*(\Cc,H_q(D))$ of the category $\Cc$ with coefficients in $H_q(D)$~\cite[Appendix~C.10]{lo-98}. By dualizing $C_{**}(D)$ we obtain the first quadrant cohomological bicomplex $C^{**}(D)$ whose total complex is quasi-isomorphic to $C^*( \hoc D; \k)$, hence a spectral sequence whose second page is given by cohomology of the category $\Cc$ with coefficients in $H^q(D)$. The next proposition summarizes the situation.

\begin{proposition}[{\cite[Theorem~18.3]{du}}]\label{pr:ss}
	For all diagrams~$D$ of shape $\Cc$ in $\top$ or in $\sset$ there are first quadrant spectral sequences
	\[
		E^{2}_{p,q}(D)=H_{p}(\Cc;H_{q}(D;\k))\Rightarrow H_{p+q}(\hoc D;\k),
	\]
	\[
		E_{2}^{p,q}(D)=H^{p}(\Cc;H^{q}(D;\k))\Rightarrow H^{p+q}(\hoc D;\k).
	\]    
\end{proposition}

\begin{remark}
	The spectral sequences of proposition~\ref{pr:ss} are the simplicial version of the spectral sequence of studied by Bousfield in~\cite{bo-87}. They go back at least to the work of Segal~\cite{se-68}, and they are a particular case of Mayer-Vietoris spectral sequence of a covering (MVSS, for short).
	To see this last point, assume that $\Cc=\Pp$ is a poset. 
	Every diagram of topological spaces $D\colon\Pp\to \top$ has a unique natural transformation to the constant diagram $*\colon\Pp\to \Top$ whose objects are single points. This natural transformation induces a continuous map $p\colon \hoc D \to \hoc *=|N_\bullet(\Pp)|$. The $0$-cells of the the latter are in one to one correspondence with the objects of $\Pp$, and for all objects $c$ we let
	\[
		\st^{\circ}(p):=\lb \sigma\in N_{\bullet}(\Pp)\colon \sigma=(p_{0}> \cdots> p_{n}),\ p_{n}=p\rb,
	\]
	be the \textit{open star} of $\Pp$. 
	The union of open simplices corresponding to $\st^{\circ}(p)$ in the realization $|N_{\bullet}(\Pp)|$ of the nerve is an open contractible subset.
	The topological space $\hoc D$ admits an open cover by $\lb U(c)\rb_{p\in \Pp}$, where
	\[
		U(p):=p^{-1}(\st^{\circ}(p))=\st^{\circ}(p)\times D(p),
	\]
	and this topological space is homotopy equivalent to $D(p)$.
	Notice that
	\[
		U(p_{0})\cap\cdots \cap 	U(p_{n})\neq\varnothing
	\]
	iff the elements $p_{i}$, $i=0,\dots,n$, are pairwise comparable in $\Pp$ (and therefore form a uniquely defined simplex in $N_{\bullet}(\Pp)$).
	Therefore, if the intersection is nonempty then it is homotopy equivalent to $D(p)$, where $p$ is the minimum of the elements $p_{0},\dots,p_{n}$.
	The nerve of the cover $\lb U(p)\rb_{p\in N_{\bullet}(\Pp)}$ coincides with the nerve $N_{\bullet}(\Pp)$ of $\Pp$.
	The respective \textit{Mayer-Vietoris spectral sequence} (e.g. see~\cite[\S 8]{bo-tu-82}) is the spectral sequence of the simplicial filtration on \v{C}ech complex
	\[
		\bigoplus\limits_{p_{0}>\dots>p_{n}\in 	\Pp}C_{q}(U(p_{0})\cap\cdots \cap U(p_{n}))\cong 
		\bigoplus\limits_{p_{0}>\dots>p_{n}\in \Pp}C_{*}(D(p_{n}))\cong C_{*,*}(D).
	\]
	Hence, MVSS of $D$ is isomorphic to $(E^*_{**}(D), d^*)$.
	(The cohomological spectral sequence case relates to the homological case in a similar way to the above.)
\end{remark}

We denote by $E^1_*(D)$ the chain complex obtained by totalizing the first page of the homological spectral sequence of proposition~\ref{pr:ss}. That is, $E^1_n(D)=\bigoplus_{p+q=n}\srep_p H_q(D)$ with differential $E^1_n(D)\to E_{n-1}^1(D)$ induced by the simplicial structure of $\srep_\bullet H_q(D)$. Similarly, we denote by $E_1^*(D)$ the chain complex obtained by totalizing the first page of the cohomological spectral sequence of proposition~\ref{pr:ss}. The following result gives an application of formality for diagrams of spaces.

\begin{proposition}\label{pr:collapse}
	Let $D$ be a diagram in $\top$ or in $\sset$. 
	\begin{enumerate}
		\item If the diagram of singular chain complexes $c\mapsto C_*(D(c);\k)$ is formal in $\mathrm{DGMod}_\k$, then the homology spectral sequence of proposition~\ref{pr:ss} collapses at page $2$. Actually, the singular chain complex $C_*( \hoc D;\k)$ is quasi-isomorphic to $E^1_*(D)$.
		\item If the diagram of singular cochain complexes $c\mapsto C^*(D(c);\k)$ is formal in $\mathrm{DGMod}_\k$, then the cohomology spectral sequence of proposition~\ref{pr:ss} collapses at page $2$. Actually, the singular cochain complex $C^*( \hoc D;\k)$ is quasi-isomorphic to $E_1^*(D)$.
	\end{enumerate}
\end{proposition}
\begin{proof}
	We prove the homological statement (the proof of the cohomological statement is similar). Formality gives a zig-zag of quasi-isomorphisms for diagrams of chain complexes connecting $H_*(D)$ with $C_*(D;\k)$. Applying simplicial replacement to this zig-zag yields a zigzag of bicomplexes:
	\begin{equation}\label{eqn:zz}
		E^1_{**}(D)\ot \dots \to C_{**}(D) 
	\end{equation}
	where $E_{**}^1(D)$ is the bicomplex with horizontal differential equal to $d^1$, and trivial vertical differential. The morphisms of bicomplexes appearing in~\eqref{eqn:zz} are columnwise quasi-isomorphisms, hence they give rise to isomorphisms of spectral sequences. Hence, collapsing of the spectral sequence of the bicomplex $C_{**}(D)$ follows from collapsing of the spectral sequence of the bicomplex $E_{**}^1(D)$. Moreover, by totalizing the morphisms of bicomplexes in~\eqref{eqn:zz} and by using lemma~\ref{lm:abut}, we obtain a zig-zag of quasi-isomorphisms of chain complexes connecting $E_*^1(D)$ with $C_*( \hoc D;\k)$.
\end{proof}

\begin{remark}\label{rk:duality}
	If $\k$ is a field, then the $\k$-linear dual of any quasi-isomorphism in $\mathrm{DGMod}_\k$ is a quasi-isomorphism. Thus if the diagram of singular chain complexes $c\mapsto C_*(D(c);\k)$ is formal in $\mathrm{DGMod}_\k$, then the diagram of singular cochain complexes  $c\mapsto C^*(D(c);\k)$ is also formal in $\mathrm{DGMod}_\k$.
\end{remark}

\begin{remark}
	By taking into account the algebraic structures discussed in remark~\ref{rem:upzigzag}, the quasi-isomorphisms of proposition~\ref{pr:collapse} can be upgraded. In the homological case, if $c\mapsto C_{*}(D(c);\k)$ is formal in $\mathrm{ACDGAlg}_{\k}$ (as is the case for abelian EML spaces over $\Q$), then the quasi-isomorphism between $C_{*}(\hoc D;\k)$ and $E_*^1(D)$ holds in $\mathrm{ACDGAlg}_{\k}$ via the Pontryagin product.
	In the cohomological case, if $c\mapsto C^{*}(D(c);\k)$ is formal in the category of strictly associative DG algebras, the quasi-isomorphism between $C^{*}(\hoc D;\k)$ and $E^{*}_1(D)$ holds in the category of strictly associative DG algebras via the Alexander-Whitney map.
	In both cases, passing to (co)homology ensures that the spectral sequence collapses and yields a strict isomorphism of graded commutative algebras between the second page of the spectral sequence and rational (co)homology of $\hoc D$.
\end{remark}

The next theorem is an application of the formality results of section~\ref{sec-formality} together with the previous considerations. Recall~\cite[II~4.3]{bo-ka-72} that a pointed connected space is called \emph{nilpotent} if its fundamental group acts nilpotently on all the homotopy groups $\pi_i$ for $i>0$. Nilpotent spaces include simply connected spaces, topological groups, as well as spaces with abelian fundamental group and contractible universal cover.

\begin{theorem}\label{thm:rcollf}
	Let $n$ be a positive integer and let $D$ be a diagram in $\top_*$ or $\sset_*$. Assume that every object $D(c)$ of the diagram is a nilpotent space such that $\pi_iD(c)\otimes\Q =0$ for all $i\ne n$ and with abelian fundamental group.
	Then the two spectral sequences of proposition~\ref{pr:ss} collapse. 
	Actually the singular chain complex $C_*(\hoc;\Q)$ is quasi-isomorphic to $E^1_*(D)$ and the singular cochain complex $C^*(\hoc;\Q)$ is quasi-isomorphic to $E_1^*(D)$.
\end{theorem}
\begin{proof}
	Since every diagram of nilpotent topological spaces can be converted into a diagram of pointed simplicial spaces with the same property on rational homotopy groups, we may assume that we work with pointed simplicial sets, i.e., we consider a diagram $D:\Cc\to \sset_*$. 
	Then~\cite[V 4.2]{bo-ka-72} provides a functorial $\Q$-localization, that is an endofunctor $L_\Q$ of nilpotent simplicial sets, together with a natural transformation $\eta_X:X\to L_\Q X$ which induces 
	an isomorphism $\pi_i(X)\otimes\Q \xrightarrow[]{\simeq}\pi_iL_\Q X$. In particular $\eta$ induces a natural transformation of diagrams $D\to L_\Q D$, such that for all objects $c$ of $\Cc$, the map $\eta_c:D(c)\to L_\Q D(c)$ induces a quasi-isomorphism on the level of singular chain complexes with coefficients in $\Q$. Moreover, $L_\Q D$ is a diagram of EML-spaces of height $n$, hence it is formal in $\mathrm{DGMod}_\Q$ by theorem~\ref{thm:formQ}. Thus the functor $\Cc\mapsto \mathrm{DGMod}_\Q$ defined by $c\mapsto C_*(D(c);\Q)$ is formal. 
	Now theorem~\ref{thm:rcollf} follows from proposition~\ref{pr:collapse} and remark~\ref{rk:duality}. 
\end{proof}

\begin{example}
	A diagram~$D\colon\Cc\to \Top$ is called a \emph{toric diagram}, if~$D(p)=(S^{1})^{d(p)}$ is a compact torus (with a fixed decomposition into product of $d(p)\in \Z_{\geq 0}$ circles) for any~$p\in \Cc$ and any arrow of~$D$ is a group homomorphism~\cite{we-zi-zi-99}.
	A toric diagram is a particular case of an EML diagram of height $1$ over $\Cc$.
\end{example}

\begin{remark}
	Let~$\alpha\colon D\to \kappa T$ be a morphism of toric diagrams, where $\kappa T$ denotes the constant diagram over $\Pp$ corresponding to a compact torus $T$.
	By using similar arguments (see theorem~\ref{thm:rcollf}) one proves the collapse of spectral sequence for $D$ in category of $H^*(BT)$-modules.
	However, this proof shows trivial abutment of such spectral sequence only for $\Q$-modules.
\end{remark}

\begin{remark}
	The height $2$ EML functorial formality in the category of bialgebras over $\Q$ implies formality in the category of $H^*(BT)$-modules.
	(After choosing a quasi-isomorphism $H^{*}(BT;\Q)\to C^*(BT;\Q)$.)
	This provides an alternative proof of Eilenberg-Moore spectral sequence collapse at page $2$ (with $\Q$ coefficients)~\cite[Theorem~4.22]{li-so-22}.
\end{remark}

\section{Acknowledgements}
The first-named author is grateful to M.~Franz for communicating an argument with non-formality over $\Z$ and $\Z/2\Z$, G.~Horel for explaining the derived approach to the problem of formality for diagrams and A.~Zakharov for drawing attention to his work on relevant mathematical problems.
G.S. would like to thank Universit\'e de Lille for its hospitality during his research visit in 2024, where the current project started.
The second-named author thanks D.~Tanr\'e for various discussions, and in particular for drawing his attention to the work of Lemaire and Sigrist~\cite{le-si-81}.

Grigory Solomadin was supported by DFG Walter Benjamin Fellowship [Deutsche Forschungsgemeinschaft (DFG) - Projektnummer 561158824].
Antoine Touz\'e acknowledges the support of the CDP C2EMPI, together with the French State under the France-2030 programme, the University of Lille, the Initiative of Excellence of the University of Lille, the European Metropolis of Lille for their funding and support of the R-CDP-24-004-C2EMPI project. 


\end{document}